\documentclass{amsart}
\usepackage{amsmath,amsthm,amssymb,amscd}
\newtoks\keywordstoks
\usepackage{enumerate}
\newtheorem{thm}{Theorem}[section]
\newtheorem{crl}[thm]{Corollary}
\newtheorem{lmm}[thm]{Lemma}
\newtheorem{prp}[thm]{Proposition}
\theoremstyle{definition}
\newtheorem{dfn}[thm]{Definition}
\theoremstyle{remark}
\newtheorem*{rem}{Remark}

\title[de Rham cohomology of diffeological spaces]{Long exact sequences for de Rham cohomology of diffeological spaces}
\author{Tadayuki Haraguchi}
\date{\today}
\subjclass[2010]{Primary 14F40; Secondary 46M18}
\keywords{diffeological space, de Rham cohomology, D-Topology}
\begin{document}
\maketitle
%%%%%%%%%%%%%%%%%%%%%%%%%%%%%%%%%%%%%%%%%%%%%%%%%%%%%%%%%%%%%%%%%%%%%%%%%%%%%%%%%%%%%%%%%%%%%%%%%%%%%%%%%%%%%%%%%%%%%%%%%%%%%%%%%%%%%%%%%%%%%%%%%%%%
%
%
% Abstruct
%
%
%%%%%%%%%%%%%%%%%%%%%%%%%%%%%%%%%%%%%%%%%%%%%%%%%%%%%%%%%%%%%%%%%%%%%%%%%%%%%%%%%%%%%%%%%%%%%%%%%%%%%%%%%%%%%%%%%%%%%%%%%%%%%%%%%%%%%%%%%%%%%%%%%%%%
\begin{abstract}
 In this paper we present the notion of de Rham cohomology with compact support for diffeological spaces.
 Moreover we shall discuss the existence of three long exact sequences.
 As a concrete example,
 we show that long exact sequences exist for the de Rham cohomology of diffeological subcartesian spaces.
\end{abstract}
%%%%%%%%%%%%%%%%%%%%%%%%%%%%%%%%%%%%%%%%%%%%%%%%%%%%%%%%%%%%%%%%%%%%%%%%%%%%%%%%%%%%%%%%%%%%%%%%%%%%%%%%%%%%%%%%%%%%%%%%%%%%%%%%%%%%%%%%%%%%%%%%%%%%
%
%
% Section1 Introduction
%
%
%%%%%%%%%%%%%%%%%%%%%%%%%%%%%%%%%%%%%%%%%%%%%%%%%%%%%%%%%%%%%%%%%%%%%%%%%%%%%%%%%%%%%%%%%%%%%%%%%%%%%%%%%%%%%%%%%%%%%%%%%%%%%%%%%%%%%%%%%%%%%%%%%%%%
\section{Introduction}
 Generally,
 the de Rham cohomology is a cohomology based on differential forms of a topological smooth manifold.
 In \cite{Sou}, 
 J.-M. Souriau introduced diffeological spaces as generalization of the notions of topological smooth manifolds.
 Moreover, 
 In \cite{Zem},
 P. Iglesias-Zemmour extended the notions of differenrial forms and de Rham cohomology groups on difeological spaces.
%
% In Section 2
% we introduce the notions of $D$-Hausdorffness, $D$-compactness, $D$-paracompactness and $D$-normality with respect to diffeological spaces.
% We define $D$-compact and $D$-paracompact with respect to subspaces of a diffeological space $X$ by using $D$-open sets of $X$.
% As a result,
 
 In Section 2 we discuss about the Hausdorffness, 
 the compactness, 
 the paracompactness,
 and the normality of $D$-topology of a diffeological space.
 Since the inclusion is not compatible with $D$-topologies,
 we need to be cautious in dealing with these notions.
 But we can show that if a diffeological space is $D$-paracompact and $D$-Hausdorff,
 then it is $D$-normal.
 In Section 3
 we introduce "diffeological subcartesian spaces" which is a diffeological space locally diffeomorphic to a (not necessarily open) subspace of an Euclidean space. % $\mathbf{R}^{n}$.
% Let $X$ be a diffeological subcartesian space.
% Then
% all $D$-open covers of $X$ have a partition of unity subordinate.
 It is shown that every diffeological subcartesian space has a partition of unity subordinate to any $D$-open cover.
 In Section 4
 we discuss de Rham cohomolgy (with compact support) in respect to diffeological spaces.
 It is shown that if there is a $D$-open cover $\{A,\,B\}$ of $X$ such that there exists a partition of unity subordinate to it,
% Let $\{A,B\}$ be a $D$-open cover of a diffeological space $X$ such that there exists a partition of unity subordinate to it.
 then we have a Mayer-Vietoris exact sequence of de Rham cohomology groups (see Theorem \ref{thm:Mayer-Vietoris 1}):
  \[
   \to  H^{p}_{\rm dR}(X) \xrightarrow{j^{\ast}_{1}\oplus j^{\ast}_{2}} H^{p}_{\rm dR}(A)\oplus H^{p}_{\rm dR}(B) \xrightarrow{i_{1}^{\ast}-i_{2}^{\ast}}  H^{p}_{\rm dR}(A \cap B)
   \xrightarrow{\delta}  H^{p}_{\rm dR}(X) \to \cdots.
  \]
 On the other hand,
 we shall see in Section 5 that
 if $X$ is $D$-Hausdorff,
 then there is a Mayer-Vietoris exact sequence of de Rham cohomology groups with compact support (see Theorem \ref{thm:Mayer-Vietoris 2}):
  \[
   \to H^{p}_{c}(A \cap B) \xrightarrow{i_{1 \ast} \oplus i_{2 \ast}} H^{p}_{c}(A) \oplus H^{p}_{c}(B) \xrightarrow{j_{1 \ast} - j_{2 \ast}} H^{p}_{c}(X) 
   \xrightarrow{\delta} H^{p+1}_{c}(A \cap B) \to \cdots.
  \]
 In particular,
 if $X$ is a diffeological subcartesian space,
 then there exist both types of Mayer-Vietoris exact sequences.
 In Section 6,
 we introduce a long exact sequence for pair of diffeological spaces.
 Let $A$ be a $D$-compact set of a diffeological subcartesian space $X$.
 If there exists a $D$-open set $M$ of $X$ such that $A$ is a deformation retract of $M$,
 then we have a long exact sequence (see Theorem \ref{thm:exact sequence}):
  \[
   \to H^{p}_{c}(X \setminus A) \xrightarrow{i_{\ast}} H^{p}_{c}(X) \xrightarrow{j^{\ast}} H^{p}_{c}(A) \xrightarrow{\delta} H^{p+1}_{c}(X \setminus A) \to \cdots.
  \] 
  
  I would like to thank Kazuhisa Shimakawa,
  who suggested me the idea of using differential $p$-forms on a diffeological space with compact support (see Definition \ref{def:definition of compact support}).
%%%%%%%%%%%%%%%%%%%%%%%%%%%%%%%%%%%%%%%%%%%%%%%%%%%%%%%%%%%%%%%%%%%%%%%%%%%%%%%%%%%%%%%%%%%%%%%%%%%%%%%%%%%%%%%%%%%%%%%%%%%%%%%%%%%%%%%%%%%%%%%%%%%%
%%%%%%%%%%%%%%%%%%%%%%%%%%%%%%%%%%%%%%%%%%%%%%%%%%%%%%%%%%%%%%%%%%%%%%%%%%%%%%%%%%%%%%%%%%%%%%%%%%%%%%%%%%%%%%%%%%%%%%%%%%%%%%%%%%%%%%%%%%%%%%%%%%%%
% Section2 Diffeological spaces
%%%%%%%%%%%%%%%%%%%%%%%%%%%%%%%%%%%%%%%%%%%%%%%%%%%%%%%%%%%%%%%%%%%%%%%%%%%%%%%%%%%%%%%%%%%%%%%%%%%%%%%%%%%%%%%%%%%%%%%%%%%%%%%%%%%%%%%%%%%%%%%%%%%%
%%%%%%%%%%%%%%%%%%%%%%%%%%%%%%%%%%%%%%%%%%%%%%%%%%%%%%%%%%%%%%%%%%%%%%%%%%%%%%%%%%%%%%%%%%%%%%%%%%%%%%%%%%%%%%%%%%%%%%%%%%%%%%%%%%%%%%%%%%%%%%%%%%%%
\section{The $D$-Topology for Diffeological spaces}
% In this section we present the notions of $D$-Hausdorffness, $D$-compactness, $D$-paracompactness and $D$-normality with respect to diffeological spaces.
% 
 A diffeological space consists of a set $X$ together with
 a family $D$ of maps from open subsets of Euclidean spaces into $X$
 satisfying the following conditions:
  \begin{description}
   \item[{\bf Covering}]
    Any constant parametrization $\mathbf{R}^n \to X$ belongs to $D$.
   \item[{\bf Locality}]
    A parametrization $P \colon U \to X$ belongs to $D$ if every point $u$ of $U$ has a neighborhood $W$ such that $P|W \colon W \to X$ belongs to $D$.
   \item[{\bf Smooth compatibility}]
    If $P \colon U \to X$ belongs to $D$, then so does the composite $P\circ Q \colon V \to X$ for any smooth map $Q \colon V \to U$ between open subsets of Euclidean spaces.
  \end{description}
 We call $D$ a diffeology of $X$, and each member of $D$ a plot of $X$.
 A map $f \colon X \to Y$ between diffeological spaces is called smooth if for any plot $P \colon U \to X$ of $X$,
 the composite $f \circ P \colon U \to Y$ is a plot of $Y$.
 Clearly, the class of diffeological spaces and smooth maps form
 a category $\bf Diff$.
  \begin{thm}[{\cite[1.60]{Zem}}, {\cite[2.1]{KKH}} ]
   The category $\bf Diff$ is complete, cocomplete, and cartesian closed.
  \end{thm}
 Let $X$ be a diffeological space.
 Let $A$ be a subset of $X$.
 We say that $A$ is $D$-open in $X$ if for any plot $P \colon U \to X$ of $X$,
 $P^{-1}(A)$ is open in $U$.
% We will denote the interior of $A$ by Int $A$.
% Then it is $D$-open in $X$.
 A subset $A$ is called $D$-closed in $X$ if for any plot $P \colon U  \to X$ of $X$,
 $P^{-1}(A)$ is closed in $U$.
 We will denote the closure of $A$ by $\overline{A}$.
% Then it is $D$-closed.
 That is to say,
 $\overline{A}$ is the smallest $D$-closed subset containing $A$.
  \begin{lmm}[{\cite[3.17]{DGE}}]\label{lmm:D-open subspace}
   Let $A$ be a $D$-open set of diffeological space $X$.
   Then a subset $B$ of $A$ is $D$-open in $X$ if and only if it is $D$-open in $A$.
  \end{lmm}
  \begin{rem}
   Let $A$ be a subset of a diffeological space $X$.
   Then we can give $A$ two topologies:
    \begin{enumerate}
     \item
      $\tau_{1}(A) \colon$the $D$-topology of the sub-diffeology on $A$;
     \item
      $\tau_{2}(A) \colon$the sub-topology of the $D$-topology on $X$.
    \end{enumerate}
   However,
   these topologies are not always the same.
   In general,
   we can only conclude that $\tau_{2}(A) \subseteq \tau_{1}(A)$.
   Therefore we need to be careful when defining separation axioms and compactness.
  \end{rem}
%%%%%%%%%%%%%%%%%%%%%%%%%%%%%%%%%%%%%%%%%%%%%%%%%%%%%%%%%%%%%%%%%%%%%%%%%%%%%%%%%%%%%%%%%%%%%%%%%%%%%%%%%%%%%%%%%%%%%%%%%%%%%%%%%%%%%%%%%%%%%%%%%%%%%%%%%%%%%%%%%%%
%%%%%%%%%%%%%%%%%%%%%%%%%%%%%%%%%%%%%%%%%%%%%%%%%%%%%%%%%%%%%%%%%%%%%%%%%%%%%%%%%%%%%%%%%%%%%%%%%%%%%%%%%%%%%%
%D-Hausdorff
%%%%%%%%%%%%%%%%%%%%%%%%%%%%%%%%%%%%%%%%%%%%%%%%%%%%%%%%%%%%%%%%%%%%%%%%%%%%%%%%%%%%%%%%%%%%%%%%%%%%%%%%%%%%%%%
%%%%%%%%%%%%%%%%%%%%%%%%%%%%%%%%%%%%%%%%%%%%%%%%%%%%%%%%%%%%%%%%%%%%%%%%%%%%%%%%%%%%%%%%%%%%%%%%%%%%%%%%%%%%%%%%%%%%%%%%%%%%%%%%%%%%%%%%%%%%%%%%%%%%%%%%%%%%%%%%%%%
 \begin{dfn}[$D$-Hausdorff space]
 A diffeological space $X$ is $D$-Hausdorff if and only if for any elements $x$ and $y$ of $X$,
 there are $D$-open neighborhoods $U_{x}$ and $U_{y}$ of $x$ and $y$,
 respectively,
 such that $U_{x} \cap U_{y}= \emptyset$.
% Then $x\not \in \overline{U}_{y}$ holds since we have $\overline{U}_{y} \subset \overline{X \setminus U_{x}}=X \setminus U_{x}$.
 \end{dfn}
 We have the following by the above remark.
  \begin{lmm}\label{lmm:D-Hausdorff}
   Let $A$ be a subspace of $X$.
   If $X$ is $D$-Hausdorff,
   then $A$ is also $D$-Hausdorff.
  \end{lmm}
%  \begin{proof}
%  For any $D$-open set $W$ of $X$,
%  $W \cap A$ is $D$-open in $A$.
%  Therefore $A$ is $D$-Hausdorff.
%   For any elements $x$ and $y$ of $A$,
%   there exist $D$-open neighborhoods $U_{x}$ and $U_{y}$ in $X$ such that $U_{x} \cap U_{y}=\emptyset$ holds.
%   Let $\tilde{U}_{x}=U_{x} \cap A$ and $\tilde{U}_{y}=U_{y} \cap A$.
%   Then $\tilde{U}_{x}$ and $\tilde{U}_{y}$ are $D$-open in $A$.
%   Therefore $A$ is $D$-Hausdorff.
%  \end{proof}
  %
  %
  %%%%%%%%%%%%%%%%%%%%%%%%%%%%%%%%%%%%%%%%%%%%%%%%%%%%%%%%%%%%%%%%%%%%%%%%%%%%%%%%%%%%%%%%%%%%%%%%%%%%%%%%%%%%%%%%%%%%%%%%%%%%%%%%%%%%%%%%%%%
  %
  %\D-compact
  %
  %%%%%%%%%%%%%%%%%%%%%%%%%%%%%%%%%%%%%%%%%%%%%%%%%%%%%%%%%%%%%%%%%%%%%%%%%%%%%%%%%%%%%%%%%%%%%%%%%%%%%%%%%%%%%%%%%%%%%%%%%%%%%%%%%%%%%%%%%%%%
  %
  %
 \begin{dfn}[$D$-compact space]
  Let $C$ be a subset of a diffeological space $X$.
  We say that $C$ is $D$-compact in $X$ if every covers of $C$ consisting of $D$-open sets of $X$ have a finite cover.
  If $X$ is $D$-compact in $X$,
  then it is called to be $D$-compact.
 \end{dfn}
  Then we have the following by Lemma \ref{lmm:D-open subspace}.
   \begin{prp}\label{prp:D-compact equivalent}
    Let $A$ be a $D$-open set of $X$.
    Then a subset $C$ of $A$ is $D$-compact in $X$ if and only if it is $D$-compact in $A$. 
   \end{prp}
  It is not difficult to prove the following.
   \begin{prp}\label{prp: subcompact}
    If $X$ is $D$-compact,
    then every $D$-closed subset of $X$ is also $D$-compact in $X$.
   \end{prp}
   \begin{prp}
    If $X$ is $D$-Hausdorff,
    then every $D$-compact subset of $X$ is $D$-closed in $X$.
   \end{prp}
   %
   %
   %%%%%%%%%%%%%%%%%%%%%%%%%%%%%%%%%%%%%%%%%%%%%%%%%%%%%%%%%%%%%%%%%%%%%%%%%%%%%%%%%%%%%%%%%%%%%%%%%%%%%%%%%%%%%%%%%%%%%%%%%%%%%%%%%%%%%
   %
   %D-paracompact
   %
   %%%%%%%%%%%%%%%%%%%%%%%%%%%%%%%%%%%%%%%%%%%%%%%%%%%%%%%%%%%%%%%%%%%%%%%%%%%%%%%%%%%%%%%%%%%%%%%%%%%%%%%%%%%%%%%%%%%%%%%%%%%%%%%%%%%%%%%
   %
   %
% \begin{dfn}[D-paracompact spaces]
  We turn to $D$-paracompactness.
  Let $X$ be a diffeological space.
  A collection $\{W_{\lambda}\}_{\lambda \in \Lambda}$ of subsets of $X$ is called locally finite if each $x \in X$ has an $D$-open neighborhood whose intersection with $W_{\lambda}$ is 
  non-empty only for finitely many $\lambda$.
% \end{dfn}
   \begin{lmm}\label{lmm:D-paracompact}
    Let $\{W_{\lambda}\}_{\lambda \in \Lambda}$ be a collection of subsets of $X$.
    If $\{W_{\lambda}\}_{\lambda \in \Lambda}$ is locally finite,
    then $\cup_{\lambda \in \Lambda} \overline{W}_{\lambda}=\overline{\cup_{\lambda \in \Lambda}W_{\lambda}}$ holds.
   \end{lmm}
   \begin{proof}
    It is clear that  $\cup_{\lambda \in \Lambda} \overline{W}_{\lambda} \subset \overline{\cup_{\lambda \in \Lambda}W_{\lambda}}$ holds.
    Conversely,
    let $x \not \in \cup_{\lambda \in \Lambda} \overline{W}_{\lambda}$.
    Then for any $\lambda \in \Lambda$,
    there exists $D$-open neighborhood $U_{\lambda}(x)$ of $x$ such that $U_{\lambda}(x) \cap W_{\lambda}= \emptyset$.
    Since $\{W_{\lambda}\}_{\lambda \in \Lambda}$ is locally finite,
    there exist a $D$-open neighborhood $V$ of $x$ and finitely many $\lambda_{i} \in \Lambda \ (1 \leq i \leq m)$ such that $V \cap W_{\lambda_{i}} \not = \emptyset$.
%    that is,
%    for any $\lambda \in \Lambda$ such that $\lambda \not = \lambda_{i}$,
%    $V \cap W_{\lambda}= \emptyset$ holds.
    Let $U=( \cap_{1 \leq i \leq m} U_{\lambda_{i}}(x)) \cap V$.
    Then $U$ is $D$-open neighborhood of $x$.
    Since for any $\lambda \in \Lambda$,
    $W_{\lambda} \cap U= \emptyset$,
    we have $(U_{\lambda \in \Lambda}W_{\lambda}) \cap U= \emptyset$.
    Therefore $x \not \in \overline{U_{\lambda \in \Lambda}W_{\lambda}}$.
   \end{proof}
  Let $\{V_{\alpha}\}_{\alpha \in I}$ and $\{U_{\beta}\}_{\beta \in J}$ be two covers of $X$.
  We say that $\{V_{\alpha}\}_{\alpha \in I}$ is a refinement of  $\{U_{\beta}\}_{\beta \in J}$ if for any $\alpha \in I$,
  there exists $\beta \in J$ such that $V_{\alpha} \subset U_{\beta}$.
   \begin{dfn}[$D$-paracompact space]
    We say that a subset $A$ of a diffeological space $X$ is $D$-paracompact in $X$ 
    if every cover of $A$ consisting of $D$-open sets of $X$ has a locally finite refinement consisting of $D$-open sets of $X$.
    If $X$ is $D$-paracompact in $X$,
    then we call it $D$-paracompact.
   \end{dfn}
   \begin{prp}
    Let $A$ be a $D$-closed subset of $X$.
    If $X$ is $D$-paracompact,
    then $A$ is $D$-paracompact in $X$.
   \end{prp}
   \begin{proof}
    Let $\{U_{\lambda}\}_{\lambda \in \Lambda}$ be a cover of $A$ consisting of $D$-open sets of $X$.
    Then ${\sf U}=\{U_{\lambda}\}_{\lambda \in \Lambda} \cup \{X \setminus A\}$ is $D$-open cover of $X$.
    Since $X$ is $D$-paracompact,
    there exists a locally finite refinement $\{V_{\alpha}\}_{\alpha \in I}$ of $\sf U$.
    Let $I^{\prime}=\{\alpha \in I|V_{\alpha} \cap A \not = \emptyset \}$.
    Then $\{V_{\alpha}\}_{\alpha \in I^{\prime}}$ is a locally finite refinement of $\{U_{\lambda}\}_{\lambda \in \Lambda}$.
   \end{proof}
%%%%%%%%%%
%%%%%%
%%%%%%%%%%%%%%%%%%%%%%%%%%%%%%%%%%%%%%%%%%%%%%%%%%%%%%%%%%%%%%%%%%%%%%%%%%%%%%%%%%%%%%%%%%%%%%%%%%%%%%%%%%%%%%%%%%%%%%%%%%%5
%
%D-nomal 
%
%%%%%%%%%%%%%%%%%%%%%%
%%%%%%%%%%%
 \begin{dfn}[$D$-normal space]
  We say that a diffeological space $X$ is $D$-normal if for any $D$-closed sets $A$ and $B$ of $X$ such that $A \cap B= \emptyset$,
  there exist $D$-open sets $U_{A}$ and $U_{B}$ of $X$ such that $A \subset U_{A},\ B \subset U_{B}$ and $U_{A} \cap U_{B}= \emptyset$.
 \end{dfn}
  From the definition,
  it is clear that we have the following.
   \begin{prp}\label{$D$-normal definition}
    A diffeological space $X$ is $D$-normal if and only if for any $D$-closed set $A$ and $D$-open set $B$ of $X$ such that $A \subset B$,
    there exists a $D$-open set $U_{A}$ of $X$ such that $A \subset U_{A} \subset \overline{U}_{A} \subset B$.
   \end{prp}
   \begin{thm}\label{thm:normal}
    If $X$ is $D$-Hausdorff and $D$-paracompact,
    then it is $D$-normal.
   \end{thm}
   \begin{proof}
    Let $x$ be an element of $X$.
    Let $F$ be a $D$-closed set of $X$ such that $x \not \in F$ and let $y$ be an element of $F$.
    Since $X$ is $D$-Hausdorff,
%    for any $y \in F$,
    there exists a $D$-open neighborhood $U_{x}$ and $U_{y}$ of $x$ and $y$,
    respectively,
    such that $U_{x} \cap U_{y}= \emptyset$.
    Then ${\sf U}=\{U_{y}|y\in F\} \cup \{X\setminus F\}$ is a $D$-open cover of $X$.
    Since $X$ is $D$-paracompact,
    there exists a locally finite refinement $\{W_{\lambda}\}_{\lambda\in \Lambda}$ of ${\sf U}$.
    Thus there are a $D$-open neighborhood $V$ of $x$ and finitely many $\lambda_{i} \ (1 \leq i \leq m)$ such that $W_{\lambda_{i}} \cap V \not = \emptyset$.
    Let $I = \{\lambda_{i}|x \not \in W_{\lambda_{i}}, 1 \leq i \leq m \}$ and $V_{0}=V \setminus \cup_{\lambda_{i} \in I} \overline{W}_{\lambda_{i}}$.
    Since $\cup_{\lambda_{i} \in I} \overline{W_{\lambda_{i}}}=\overline{\cup_{\lambda_{i} \in I} W_{\lambda_{i}}}$ holds,
    $V_{0}$ is a $D$-open neighborhood of $x$.
%    Then we have $V_{0} \cap W = \emptyset$.
    Let $J=\{\lambda \in \Lambda|W_{\lambda} \cap F \not = \emptyset\}$ and $W=\cup_{\lambda \in J}W_{\lambda}$.
    Then $W$ is a $D$-open set of $X$ such that $F \subset W$.
%    Let $V_{0}=V \setminus \cup_{\lambda \in \Lambda} \overline{W}_{\lambda}=V \setminus \overline{\cup_{\lambda \in \Lambda} W_{\lambda}}$ by Lemma \ref{lmm:D-paracompact}.
%    Then $V_{0}$ is a $D$-open neighborhood of $x$ since for each $\lambda \in \Lambda$,
%    $x \not \in \overline{W}_{\lambda}$ holds.
    Then we have $V_{0} \cap W = \emptyset$.
    
    Next,
    let $A$ and $B$ be $D$-closed sets of $X$ such that $A \cap B = \emptyset$.
    By the above condition,
    for any $a \in A$,
    there exist a $D$-open neighborhood $U_{a}$ of $a$ and $D$-open set $W(a)$ of $X$ such that $B \subset W(a)$ and $U_{a} \cap W(a)= \emptyset$.
    Let ${\sf U}^{\prime}=\{U_{a}|a \in A\} \cup \{X \setminus A\}$.
    Then $\sf U^{\prime}$ is a $D$-open cover of $X$.
    Since $X$ is $D$-paracompact,
    there exists a locally finite refinement $\{L_{\lambda}\}_{\lambda \in \Lambda}$ of $\sf U^{\prime}$.
    For any $b \in B$,
    there exist a $D$-open neighborhood $M_{b}$ of $b$ and finitely many $\lambda_{i} \ (1 \leq i \leq m)$ such that $M_{b} \cap L_{\lambda_{i}} \not = \emptyset$.
    Let $I=\{\lambda_{i}|a \not \in L_{\lambda_{i}},1 \leq i \leq m \}$ and let $M= \cup_{b \in B} \tilde{M}_{b}$,
    where $\tilde{M}_{b}=M_{b} \setminus \cup_{\lambda_{i} \in I} \overline{L_{\lambda_{i}}}$ is a $D$-open neighborhood of $b$.
    Let $J=\{\lambda \in \Lambda|L_{\lambda} \cap A \not = \emptyset\}$ and let $L=\cup_{\lambda \in J} L_{\lambda} \supset A$.
    Then we have $L \cap M= \emptyset$.
   \end{proof}
%%%%%%%%%%%%%%%%%%%%%%%%%%%%%%%%%%%%%%%%%%%%%%%%%%%%%%%%%%%%%%%%%%%%%%%%%%%%%%%%%%%%%%%%%%%%%%%%%%%%%%%%%%%%%%%%%%%%%%%%%%%%%%%%%%%%%%%%%%%%%%%%%%%%%%%%%%%
%
%
% Section 3 Diffeological subcartesian spaces
%
%
%%%%%%%%%%%%%%%%%%%%%%%%%%%%%%%%%%%%%%%%%%%%%%%%%%%%%%%%%%%%%%%%%%%%%%%%%%%%%%%%%%%%%%%%%%%%%%%%%%%%%%%%%%%%%%%%%%%%%%%%%%%%%%%%%%%%%%%%%%%%%%%%%%%%%%%%%%%%
\section{Diffeological subcartesian spaces}
 In this section we present the notion of diffeological subcartesian spaces.
 Moreover we prove that every diffeological subcartesian space has a partition of unity subordinate to arbitrary $D$-open over.
% all $D$-open covers of diffeological subcartesian space have a partition of unity subordinate.
 %
 \begin{dfn}[diffeological subcartesian space]
 We say that a diffeological space $X$ is a diffeological subcarteisan space if the following conditions are satisfied.
  \begin{enumerate}
   \item
    $X$ is $D$-Hausdorff and $D$-paracompact.
   \item
    For each $x$ in $X$,
    there exists a diffeomorphism $\varphi_{x} \colon U \to U^{\prime}$,
    called a chart at $x$,
    from a $D$-open neighborhood $U$ of $x$ to a subspace $U^{\prime}$ of an Euclidean space $\mathbf{R}^{n_{x}}$,
    where $n_{x} \geq 0$ and $U^{\prime}$ need not be open in $\mathbf{R}^{n_{x}}$.
  \end{enumerate}
% We call $\varphi_{x}$ a chart.
 It is clear that a diffeological subcartesian space is $D$-normal by Theorem \ref{thm:normal}.
 \end{dfn}
%%%%%%%%%%%%%%%%%%%%%%%%%%%%%%%%%%%%%%%%%%%%%%%%%%%%%%%%%%%%%%%%%%%%%%%%%%%%%%%%%%%%%%%%%%%%%%%%%%%%%%%%%%%%%%%%%%%%%%%%%%%%%%%%%%%%%%%%%%%%%%%%%%%%%%%%%%%%%
%%
%    Shirink Lemma
%%
%%%%%%%%%%%%%%%%%%%%%%%%%%%%%%%%%%%%%%%%%%%%%%%%%%%%%%%%%%%%%%%%%%%%%%%%%%%%%%%%%%%%%%%%%%%%%%%%%%%%%%%%%%%%%%%%%%%%%%%%%%%%%%%%%%%%%%%%%%%%%%%%%%%%%%%%%%%%%
  \begin{lmm}[Shrink Lemma]
   If a diffeological space $X$ is $D$-Hausdorff and $D$-paracompact,
   then for every $D$-open cover $\{U_{\lambda}\}_{\lambda \in \Lambda}$ of $X$,
   there exists a locally finite $D$-open cover $\{W_{\lambda}\}_{\lambda \in \Lambda}$ such that  $\overline{W}_{\lambda} \subset U_{\lambda}$ holds for any $\lambda \in \Lambda$.
%   $\overline{W}_{\lambda} \subset U_{\lambda}$ holds.
  \end{lmm}
  \begin{proof}
   Let $\{U_{\lambda}\}_{\lambda \in \Lambda}$ be a $D$-open cover of $X$.
   Since $\cup_{\lambda \in \Lambda}U_{\lambda}=X$ holds,
   we have
    \[
     \cap_{\lambda  \in \Lambda} (X \setminus U_{\lambda})  
      = \cap_{\lambda \not = \lambda^{\prime} \in \Lambda}( X \setminus U_{\lambda^{\prime}}) \cap (X \setminus U_{\lambda})= \emptyset.
    \] 
   Since $X$ is $D$-normal by Theorem \ref{thm:normal},
   there exist two $D$-open sets $S_{\lambda}$ and $W$ of $X$ such that
    \[
     \cap_{\lambda \not = \lambda^{\prime} \in \Lambda} (X \setminus U_{\lambda^{\prime}} ) \subset S_{\lambda},\
     X \setminus U_{\lambda} \subset W \ {\rm and} \ S_{\lambda} \cap W= \emptyset.
    \]
   Thus we have $\overline{S}_{\lambda} \subset U_{\lambda}$ and $S_{\lambda} \cup (\cup_{\lambda \not = \lambda^{\prime} \in \Lambda} X \setminus U_{\lambda^{\prime}})=X$
   since $S_{\lambda} \subset X \setminus W \subset U_{\lambda}$ holds.
   Let ${\sf U}=\{S_{\lambda} | \lambda \in \Lambda\}$.
   Then $\sf U$ is a $D$-open cover of $X$.
   Since $X$ is $D$-paracompact,
   there exists a locally finite refinement $\{ W_{j} \}_{j \in J}$ of $\sf U$.
   For any $\lambda \in \Lambda$,
   let $J_{\lambda}=\{j_{\lambda} \in J|W_{j_{\lambda}} \subset S_{\lambda}\}$ and let $V_{\lambda}=\cup_{j_{\lambda} \in J_{\lambda}}W_{j_{\lambda}}$.
%   By Lemma \ref{lmm:D-paracompact},
   Then we have
    \[
     \overline{V}_{\lambda}=\overline{\cup_{j_{\lambda} \in J_{\lambda}}W_{j_{\lambda}}} \subset
     %\cup_{j_{\lambda} \in J_{\lambda}} \overline{W}_{j_{\lambda}} 
     \overline{S_{\lambda}} \subset U_{\lambda}.
    \]
   Then $\{V_{\lambda}|\lambda \in \Lambda\}$ is a locally finite $D$-open cover of $X$.
  \end{proof}
   Let $X$ be a diffeological space.
 If $f \colon X \to \mathbf{R}$ is a real-valued smooth map on $X$,
 the support of $f$,
 denoted by supp$f$,
 is the closure of the set of points where $f$ is nonzero:
  \[
   {\rm supp}f=\overline{ \{p \in X;f(p) \not =0\}}.
  \]
 Let ${\sf G}=\{A_{\lambda} \}_{\lambda \in \Lambda}$ be an arbitrary $D$-open cover of $X$.
 We say that a collection $\{\phi_{\lambda} \colon X \to \mathbf{R} \}_{\lambda \in \Lambda}$ is a partition of unity subordinate to $\sf G$ if the following conditions are satisfied:
  \begin{enumerate}
   \item 
    $0 \leq \phi_{\lambda}(x) \leq 1$ for all $\lambda \in \Lambda$ and all $x \in X$,
   \item 
    {\rm supp}$\phi_{\lambda} \subset A_{\lambda}$,
   \item 
    the set $\{\rm supp \phi_{\lambda}\}_{\lambda \in \Lambda}$ of supports is locally finite, and
   \item 
    $\sum_{\lambda \in \Lambda} \phi_{\lambda}(x)=1$ for all $x\in X$.
  \end{enumerate}
%%%%%%%%%%%%%%%%%%%%%%%%%%%%%%%%%%%%%%%%%%%%%%%%%%%%%%%%%%%%%%%%%%%%%%%%%%%%%%%%%%%%%%%%%%%%%%%%%%%%%%%%%%%%%%%%%%%%%%%%%%%%%%%%%%%%%%%%%%%%%%%%%%%%%%%%%%%%
%
% Partition of unity subcordinate 
%
%%%%%%%%%%%%%%%%%%%%%%%%%%%%%%%%%%%%%%%%%%%%%%%%%%%%%%%%%%%%%%%%%%%%%%%%%%%%%%%%%%%%%%%%%%%%%%%%%%%%%%%%%%%%%%%%%%%%%%%%%%%%%%%%%%%%%%%%%%%%%%%%%%%%%%%%%%%
  \begin{thm}\label{thm:partition of unity}
   Let $X$ be a diffeological subcartesian space.
   Then
   for any $D$-open cover ${\sf U}$ of $X$,
   there exists a partition of unity subordinate to $\sf U$.
  \end{thm}  
  \begin{proof}
   Let ${\sf U}=\{U_{\lambda}\}_{\lambda \in \Lambda}$ be a $D$-open cover of $X$.
   Without loss of generality we may assume that the elements of $\sf U$ are chart domains.
   By Shrink Lemma,
   there exist locally finite $D$-open covers $\{V_{\lambda}\}$ and $\{W_{\lambda}\}$ of $X$ such that
    \[
     \overline{V}_{\lambda} \subset W_{\lambda} \subset U_{\lambda}.
    \]
%   Then there exists an open set $\tilde{U}_{\lambda}$ of $\mathbf{R}^{n_{\lambda}}$ such that $\varphi_{\lambda}^{-1}(\tilde{U}_{\lambda})=U_{\lambda}$,
%   where $\varphi_{\lambda}$ is chart.
   Let $\varphi_{\lambda} \colon U_{\lambda} \to U^{\prime}_{\lambda}$ be a chart for each $\lambda \in \Lambda$ and $U^{\prime}_{\lambda}$ is the subset of $\mathbf{R}^{n_{\lambda}}$.
   Then there are a closed subset $\tilde{V}_{\lambda}$ and an open subset $\tilde{W}_{\lambda}$ in $\mathbf{R}^{n_{\lambda}}$
   such that $\varphi^{-1}_{\lambda}(\tilde{V}_{\lambda})=\overline{V}_{\lambda}$ 
   and $\varphi^{-1}_{\lambda}(\tilde{W}_{\lambda})=W_{\lambda}$,
   respectively.
   By {\cite[2.19]{Lee}},
   there exists a smooth function
   $g_{\lambda} \colon \mathbf{R}^{n_{\lambda}} \to \mathbf{R}$
   such that supp$g_{\lambda} \subset \tilde{W}_{\lambda},\ g_{\lambda}|_{\tilde{V}_{\lambda}} \equiv 1$ and $g_{\lambda}|_{\mathbf{R}^{n_{\lambda}} \setminus \tilde{W}_{\lambda}} \equiv 0$.
%
%
%   Let $\varphi_{\lambda} \colon U_{\lambda} \to \tilde{U}_{\lambda}$ be the chart of $X$,
%   where $\tilde{U}_{\lambda}$ is an open subset of $\mathbf{R}^{n_{\lambda}}$.
%   Then there are a closed set $\tilde{V}_{\lambda}$ and an open set $\tilde{W}_{\lambda}$ of $\tilde{U}_{\lambda}$ such that 
%   $\varphi^{-1}_{\lambda}( \tilde{V}_{\lambda})=\overline{V}_{\lambda}$ and $\varphi^{-1}_{\lambda}(\tilde{W}_{\lambda})=W_{\lambda}$,
%  respectively.
%  By {\cite[2.19]{Lee}},
%  there exists a smooth function $g_{\lambda} \colon \tilde{U}_{\lambda} \to \mathbf{R}$ such that $g_{\lambda}|\tilde{V}_{\lambda} \equiv 1$ and 
%  supp($g_{\lambda}$)$\subset \tilde{W}_{\lambda}$.
  We define $f_{\lambda} \colon X \to \mathbf{R}$ by 
   \begin{eqnarray}
    f_{\lambda}(x)=\left\{
     \begin{array}{ll}
      g_{\lambda} \circ \varphi_{\lambda}(x) & x \in U_{\lambda} \nonumber \\
      0 & x \in X \setminus U_{\lambda} .\nonumber
     \end{array}
     \right.
    \end{eqnarray}
   Define new function $\phi_{\lambda} \colon X \to \mathbf{R}$ by
    \[
     \phi_{\lambda}(x)= \frac{f_{\lambda}(x)}{\sum_{\lambda^{\prime} \in \Lambda}f_{\lambda^{\prime}}(x)}.
    \]
   Then $\{\phi_{\lambda} \colon X \to \mathbf{R} | \lambda \in \Lambda\}$ is a partition of unity subordinate to $\sf U$.
  \end{proof}
  \begin{crl}\label{crl:partition}
    Let $A$ be a $D$-closed subset of a diffeological subcartesian space $X$.
    Let $U_{A}$ be a $D$-open subset of $X$ containing $A$.
    Then there exists a function $\varphi \colon X \to \mathbf{R}$ such that {\rm supp}$\varphi \subset U_{A}$ and $\varphi \equiv 1$ on $A$.
   \end{crl}
   \begin{proof}
    Since $\{U_{A},\, X \setminus A\}$ is a $D$-open cover of $X$,
    there exists a partition of unity subordinate $\{\varphi_{U_{A}},\, \varphi_{X\setminus A}\}$ to $\{U_{A},\, X \setminus A\}$ by Theorem \ref{thm:partition of unity}.
    Since $\varphi_{X \setminus A} \equiv 0$ on $A$,
    the function $\varphi_{U_{A}}$ has the required properties.
   \end{proof}
%%%%%%%%%%%%%%%%%%%%%%%%%%%%%%%%%%%%%%%%%%%%%%%%%%%%%%%%%%%%%%%%%%%%%%%%%%%%%%%%%%%%%%%%%%%%%%%%%%%%%%%%%%%%%%%%%%%%%%%%%%%%%%%%%%%%%%%%%%%%%%%%%%%%%%%%%%%%%
%
%
% Section 3 de Rham cohomology
%
%
%%%%%%%%%%%%%%%%%%%%%%%%%%%%%%%%%%%%%%%%%%%%%%%%%%%%%%%%%%%%%%%%%%%%%%%%%%%%%%%%%%%%%%%%%%%%%%%%%%%%%%%%%%%%%%%%%%%%%%%%%%%%%%%%%%%%%%%%%%%%%%%%%%%%%%%%%%%%
\section{de Rham cohomology of diffeological spaces}
 In this section we shall show that there exists a Mayer-Vietoris exact sequence with respect to de Rham cochomology of diffeological spaces.
 
 We first recall from \cite{Zem} the notion of differential forms on a diffeological space.
 A covariant antisymmetric $p$-tensor of $\mathbf{R}^{n}$ {\cite[6.11]{Zem}} is called a linear $p$-form of $\mathbf{R}^{n}$.
 The vector space of linear $p$-forms of $\mathbf{R}^{n}$ is denoted by $\Lambda^{p}( \mathbf{R}^{n})$.
% We regard $\Lambda^{p}( \mathbf{R}^{n})$ as a subspace of 
%  \[
%   C^{\infty}(\mathbf{R}^{n} \times \mathbf{R}^{n} \times \cdots \mathbf{R}^{n}, \mathbf{R}).
%  \]  
  Let $U$ be an open set of $\mathbf{R}^{n}$.
  Let $d$ be a linear map from $C^{\infty}(U, \Lambda^{p}(\mathbf{R}^{n}))$ to $C^{\infty}(U,\Lambda^{p+1}(\mathbf{R}^{n}))$ defined by {\cite[6.24]{Zem}}.
  Then we have $d \circ d=0$.
 \begin{dfn}[{\cite[6.28]{Zem}}]
  Let $X$ be a diffeological space.
  We say that $\alpha$ is a differential $p$-form on $X$ if the following two conditions are fulfilled
   \begin{enumerate}
    \item
     For all integers $n$,
     for all $n$-plots $P \colon U \to X$,
     we have
      \[
       \alpha(P) \in C^{\infty}(U, \Lambda^{k}( \mathbf{R}^{n})).
      \]
    \item
     For all open sets $V$ of $\mathbf{R}^{m}$,
     $m \geq 0$,
     for all smooth parametrizations $F \colon V \to U$,
     we have
      \[
       \alpha(P \circ F)= F^{\ast}( \alpha(P)),
      \]
   \end{enumerate}
  where $F^{\ast}(\alpha(P))$ (cf.\ {\cite[6.22]{Zem}}) is defined by% for all $v \in V$ and for all $k$-vectors $x_{1}, \cdots , x_{p} \in \mathbf{R}^{m}$,
   \[
     F^{\ast}( \alpha(P))(v)(x_{1}) \cdots (x_{p})= \alpha(P)(F(v))(D(F)(v)(x_{1})) \cdots (D(F)(v)(x_{p}))
   \]
  for all $v \in V$ and for all $k$-vectors $x_{1}, \cdots , x_{p} \in \mathbf{R}^{m}$.
%  for all $v \in V$ and for all $k$-vectors $x_{1}, \cdots , x_{p} \in \mathbf{R}^{m}$.
 \end{dfn}
  The condition $\alpha(P \circ F)=F^{\ast}(\alpha(P))$ is called the smooth compatibility condition,
  The set of differential $p$-forms on $X$ is clearly a real vector space,
  and will be denoted by $\Omega^{p}(X)$.
%  Let $A$ be a $D$-open set of $X$.
%  For any $\alpha \in \Omega^{p}(X)$,
%  we define $\alpha|_{A}$ by for any plot $P \colon U \to X$ of $X$,
%   \[
 %   \alpha|_{A}(P)= \alpha(P|_{P^{-1}(A)}).
 %  \]
 % Then $\alpha|_{A}$ is an element of $\Omega^{p}(A)$.
  We define a linear map $d$ from $\Omega^{p}(X)$ to $\Omega^{p+1}(X)$ by
   \[
    (d \alpha)(P)=d( \alpha(P))
   \]
  for any $\alpha \in \Omega^{p}(X)$ and any plot $P \colon U \to X$ of $X$.
  Clearly,
  $d \circ d=0$ holds.
  Therefore we have a cochain complex $\{ \Omega^{p}(X),\, d\}$,
  called the de Rham complex.
  We define 
   \begin{eqnarray}
    Z^{p}(X) &=& {\rm Ker}[d \colon \Omega^{p}(X) \to \Omega^{p+1}(X)] \ {\rm and} \nonumber \\
    B^{p}(X) &=& {\rm Im}[d \colon \Omega^{p-1}(X) \to \Omega^{p}(X)]. \nonumber
   \end{eqnarray}
  Since $B^{p}(X)$ is a linear subspace of $Z^{p}(X)$,
  we can define the $p$-th de Rham cohomology group of $X$ to be the quotient vector space
   \[
    H^{p}_{dR}(X)=Z^{p}(X)/B^{p}(X).
   \]
 Let $f \colon X \to Y$ be a smooth map between diffeological spaces.
 We define 
  \[
   f^{\ast} \colon \Omega^{p}(Y) \to \Omega^{p}(X)
  \]
 by 
   $f^{\ast}( \alpha )(P)= \alpha (f \circ P)$
 for any $\alpha \in \Omega^{p}(Y)$ and any plot $P$ of $X$.
 Then we have $f^{\ast}(d \alpha)=d(f^{\ast} \alpha)$ {\cite[6.25]{Zem}}.
 Thus $f$ induces a homomorphism $f^{\ast}\colon H^{p}_{\rm dR}(Y) \to H^{p}_{\rm dR}(X)$.
 Let $A$ be a $D$-open set of $X$.
  For any $\alpha \in \Omega^{p}(X)$,
  we define $\alpha|_{A} = i^{\ast}_{A}(\alpha) \in \Omega^{p}(A)$,
  where $i_{A} \colon A \to X$ is the inclusion map.
%Therefore we have a contravariant functor $H^{\ast}_{\rm dR}$ from the category of diffeological spaces to the category of Abelian groups.
  \begin{prp}\label{prp:coprod}
   Let $X$ be a diffeological space.
   Let $\{ X_{\lambda} \}$ be a collection of subspaces of $X$ such that $X= \coprod_{\lambda \in \Lambda} X_{\lambda}$.
   Then $H^{p}_{\rm dR}(X) \ and \ \prod_{\lambda \in \Lambda} H^{p}_{\rm dR}(X_{\lambda})$ are isomorphic for each $p$.
  \end{prp}
  \begin{proof}
 It is clear that for each $\lambda \in \Lambda$,
 $X_{\lambda}$ is $D$-open in $X$ by the definition of coproduct spaces.
 We define a homomorphism 
  \[
   \prod i_{\lambda}^{\ast} \colon \Omega^{p}(X) \to \prod_{\lambda \in \Lambda} \Omega^{p}(X_{\lambda})
  \]
 by %for any $\omega \in \Omega^{p}(X),\ \prod i^{\ast}_{\lambda}(\omega)=(i^{\ast}_{\lambda}(\omega))_{\lambda \in \Lambda}$,
 $\prod i^{\ast}_{\lambda}(\omega)=(i^{\ast}_{\lambda}(\omega))_{\lambda \in \Lambda}$
 for any $\omega \in \Omega^{p}(X)$,
 where $i_{\lambda}^{\ast} \colon \Omega^{p}(X) \to \Omega^{p}(X_{\lambda})$ is the cochain map induced by the inclusion $i_{\lambda}^{\ast} \colon X_{\lambda} \to X$.
 Let $\omega$ be an element of $\mbox{Ker}(\prod i^{\ast}_{\lambda})$.
 Since $(\prod i^{\ast}_{\lambda})(\omega)=((i^{\ast}_{\lambda}(\omega))_{\lambda \in \Lambda}=(\omega|_{X_{\lambda}})_{\lambda \in \Lambda}=0$,
 we have $\omega=0$.
 Let $(\tau_{\lambda})_{\lambda \in \Lambda}$ be an element of $\prod_{\lambda \in \Lambda} \Omega^{p}(X_{\lambda})$.
 We define $\tau \in \Omega^{p}(X)$ by for any $\lambda \in \Lambda,\ \tau|_{X_{\lambda}}=\tau_{\lambda}$.
 Then we have $(\prod i^{\ast}_{\lambda})(\tau)=(\tau_{\lambda})_{\lambda \in \Lambda}$
 since $X_{\lambda} \cap X_{\lambda^{\prime}}$ is empty for each $\lambda \not = \lambda^{\prime}$.
  \end{proof}

Let $A$ and $B$ be two $D$-open sets of $X$ such that $X=A \cup B$ holds.
Then we have a diagram:
%\begin{eqnarray}
% \begin{picture}(160,100)
%  \put(0,50){$A \cap B$}
%  \put(60,0){$B$}
%  \put(100,50){$X=A \cup B$}
%  \put(60,100){$A$}
%  \put(30,81){$i_{1}$}
%  \put(30,21){$i_{2}$}
%  \put(90,21){$j_{2}$}
%  \put(90,81){$j_{1}$}
%  \put(22,62){\vector(1,1){35}}
%  \put(22,45){\vector(1,-1){35}}
%  \put(70,96){\vector(1,-1){35}}
%  \put(70,10){\vector(1,1){35}}
% \end{picture} 
%\end{eqnarray}
 \begin{eqnarray}
  \begin{CD}
   A \cap B 
   @>i_{1}>>
   A \\
   @V i_{2} VV
   @V j_{1} VV
   \\
   B 
   @>j_{2}>>
   X=A \cup B
  \end{CD}
 \end{eqnarray}
 consisting of inclusions.
 Now, 
 consider the following sequence:
\begin{eqnarray}
 0
 \to
 \Omega^{p}(X)
 \xrightarrow{j_{1}^{\ast} \oplus j_{2}^{\ast}}
 \Omega^{p}(A) \oplus \Omega^{p}(B)
 \xrightarrow{i_{1}^{\ast}-i_{2}^{\ast}}
 \Omega^{p}(A \cap B)
 \to
 0,
% \begin{CD}
%  0@>>>
%  \Omega^{p}(X)
%  @> j_{1}^{\ast} \oplus j_{2}^{\ast} >>
%  \Omega^{p}(A) \oplus \Omega^{p}(B)
%  @> i_{1}^{\ast}-i_{2}^{\ast} >>
%  \Omega^{p}(A \cap B)
%  @>>>0
% \end{CD},
\end{eqnarray}
 where
\[
 (j_{1}^{\ast} \oplus j_{2}^{\ast})(\omega)=(j_{1}^{\ast}(\omega),j_{2}^{\ast}(\omega)),\
 (i_{1}^{\ast}-i_{2}^{\ast})(\omega)=i_{1}^{\ast}(\omega)-i_{2}^{\ast}(\omega).
\]
 Then we have the following.
 %%%%%%%%%%%%%%%%%%%%%%%%%%%%%%%%%%%%%%%%%%%%%%%%%%%%%%%%%%%%%%%%%%%%%%%%%%%%%%%%%%%%%%%%%%%%%%%%%%%%%%%%%%%%%%%%%%%%%%%%%%%%%%%%%%%%%%%%%%%%%%%%%%%%%%%%
 %
 % Mayer-Vietoris
 %
 %%%%%%%%%%%%%%%%%%%%%%%%%%%%%%%%%%%%%%%%%%%%%%%%%%%%%%%%%%%%%%%%%%%%%%%%%%%%%%%%%%%%%%%%%%%%%%%%%%%%%%%%%%%%%%%%%%%%%%%%%%%%%%%%%%%%%%%%%%%%%%%%%%%%%%%%%%
 \begin{thm}[Mayer-Vietoris exact sequence]\label{thm:Mayer-Vietoris 1}
  Let $X$ be a diffeological space.
  Let $\{A,\,B\}$ be a $D$-open cover of $X$.
  If there exists a partition of unity $\varphi_{i} \colon X \to \mathbf{R} \ (i =A,\,B)$ subordinate to $\{A,\,B\}$,
%partition of unity subordinate $\{\varphi_{i} \colon X \to \mathbf{R}| i=A,B\}$ to $\{A,B\}$,
  then we have a long exact sequence:
  \[
   \to H^{p}_{\rm dR}(X) \xrightarrow{j^{\ast}_{1}\oplus j^{\ast}_{2}}  H^{p}_{\rm dR}(A)\oplus H^{p}_{\rm dR}(B) \xrightarrow{i_{1}^{\ast}-i_{2}^{\ast}}  H^{p}_{\rm dR}(A \cap B) 
   \xrightarrow{\delta} H^{p+1}_{\rm dR}(X) \to \cdots.
  \] 
%    \[
%   \begin{CD}
%    \dots @> \delta >> 
%    H^{p}_{\rm dR}(X)
%    @>j^{\ast}_{1}\oplus j^{\ast}_{2}>>
%    H^{p}_{\rm dR}(A)\oplus H^{p}_{\rm dR}(B)
%    @>i_{1}^{\ast}-i_{2}^{\ast}>>
%    H^{p}_{\rm dR}(A \cap B) 
%    @.@.
%    @> \delta >>
%    H^{p+1}_{\rm dR}(X)  
%    @>j^{\ast}_{1}\oplus j^{\ast}_{2}>> .
%   \end{CD}
%  \]
 \end{thm}
 \begin{proof}
  To see existence of the Mayer-Vietoris exact sequence,
 it suffices to show that the sequence (2) is exact for each $p$.
 We shall show that $j_{1}^{\ast} \oplus j_{2}^{\ast}$ is injective.
 Let $\alpha$ be an element of Ker$(j_{1}^{\ast} \oplus j_{2}^{\ast})$.
 Since $\alpha|_{A}=0=\alpha|_{B}$,
 we have $\alpha=0$.
% For any plot $P \colon U \to X$ of $X$,
%  \[
%   \alpha(P|U_{A})=\alpha(P|U_{B})=0,
%  \]
% where $U_{A}=P^{-1}(A)$ and $U_{B}=P^{-1}(B)$ are open sets in $U$.
% Thus $\alpha(P)=0$.
%
 Let $(j_{1}^{\ast} \oplus j_{2}^{\ast})(\omega)$ be an element of Im($j_{1}^{\ast} \oplus j_{2}^{\ast}$).
 Since $i^{\ast}_{1}j^{\ast}_{1}(\alpha)=\alpha|_{A\cap B}=i_{2}^{\ast} j_{2}^{\ast}(\alpha)$,
 we have
 \begin{eqnarray}
   (i_{1}^{\ast}-i_{2}^{\ast})\circ (j_{1}^{\ast} \oplus j_{2}^{\ast})(\omega)
   =i_{1}^{\ast}j_{1}^{\ast}(\omega)-i_{2}^{\ast}j_{2}^{\ast}(\omega)
   = 0. \nonumber
  \end{eqnarray}
% For any plot $P \colon U \to A \cap B$ of $A \cap B$,
% we have
%  \begin{eqnarray}
%   (i_{1}^{\ast}-i_{2}^{\ast})\circ (j_{1}^{\ast} \oplus j_{2}^{\ast})(\omega)(P)
%   =(i_{1}^{\ast}j_{1}^{\ast}(\omega)-i_{2}^{\ast}j_{2}^{\ast}(\omega))(P) 
%   = 0. \nonumber
%  \end{eqnarray}
 Thus Im$(j_{1}^{\ast} \oplus j_{2}^{\ast}) \subset {\rm Ker}(i_{1}^{\ast}-i_{2}^{\ast})$. 
 Let $(\alpha, \beta)$ be an element of ${\rm Ker}(i_{1}^{\ast}-i_{2}^{\ast})$.
% For any plot $P \colon U \to X$,
 We define $\omega \in \Omega^{p}(X)$ by %for any plot $P:U \longrightarrow X$ of $X$,
  \begin{eqnarray}
   \omega=\left\{
    \begin{array}{ll}
     \alpha & {\rm on} \ A \\ \\
     \beta & {\rm on} \ B, \nonumber\\ 
    \end{array}
   \right.
  \end{eqnarray}
  that is,
  for any plot $P \colon U \to X$of $X$,
  $\omega(P)=\alpha(P|_{P^{-1}(A)})$ on $P^{-1}(A)$ and $\omega(P)=\beta(P|_{P^{-1}(B)})$ on $P^{-1}(B)$.
 Then $\omega$ is well-defined since $\alpha|_{A \cap B}=\beta|_{A \cap B}$ holds.
 Clearly,
 we have
 $(j_{1}^{\ast} \oplus j_{2}^{\ast})(\omega)=(\alpha,\beta)$.
% It is clear that $\omega(P)$ is well defined by $i_{1}^{\ast}(\alpha)=i_{2}^{\ast}(\beta)$.
% Then we have $(j_{1}^{\ast} \oplus j_{2}^{\ast})(\omega)=(\alpha,\beta)$.
 Thus Ker$(i_{1}^{\ast}-i_{2}^{\ast}) \subset {\rm Im}(j_{1}^{\ast} \oplus j_{2}^{\ast})$.
 We shall show that $(i_{1}^{\ast}-i^{\ast}_{2})$ is surjective.
 Let $\sigma$ be an element of $\Omega^{p}(A\cap B)$.
 Since supp$\varphi_{B} \cap A \subset A \cap B$,
 we can define $\eta_{A} \in \Omega^{p}(A)$ by 
  \begin{eqnarray}
   \eta_{A}=
    \left\{
     \begin{array}{ll}
      \varphi_{B} \times \sigma & on \ A \cap B \nonumber \\
      0 & on \ A \setminus {\rm supp}\varphi_{B}, \nonumber
     \end{array}
    \right.
  \end{eqnarray}
 that is,
 for any plot $P \colon V \to A$ of $A$,
 $\eta_{A}(P)=(\varphi_{B} \circ P|_{P^{-1}(A \cap B)}) \times \sigma(P|_{P^{-1}(A\cap B)})$ on $ P^{-1}(A \cap B)$ and $\eta_{A}(P)=0$ on $ P^{-1}(A \setminus {\rm supp}\varphi_{B})$.
 Then $\eta_{A}$ satisfies the smooth compatibility condition $F^{\ast}(\eta_{A}(P))=\eta_{A}(P \circ F)$
 for every smooth map $F$
 from an open set of Euclidean spaces to the domain of $P$.
 Similarly,
 we define $\eta_{B} \in \Omega^{p}(B)$ by 
  \begin{eqnarray}
   \eta_{B}=
    \left\{
     \begin{array}{ll}
      -\varphi_{A} \times \sigma & on \ A \cap B \nonumber \\
      0 & on \ A \setminus {\rm supp}\varphi_{A}. \nonumber
     \end{array}
    \right.
  \end{eqnarray}
 Then we have $(i_{1}^{\ast}-i_{2}^{\ast})(\eta_{A},\eta_{B})=\varphi_{B} \times \sigma + \varphi_{A} \times \sigma=(\varphi_{B}+\varphi_{A}) \times \sigma=\sigma$.
 Therefore $(i_{1}^{\ast}-i_{2}^{\ast})$ is surjective.
 \end{proof}
%
%
%%%%%%%%%%%%%%%%%%%%%%%%%%%%%%%%%%%%%%%%%%%%%%%%%%%%%%%%%%%%%%%%%%%%%%%%%%%%%%%%%%%%%%%%%%%%%%%%%%%%%%%%%%%%%%%%%%%%%%%%%%%%%%%%%%%%%%%%%%%%%%%%%%%%%%%%%%
%
%
% Section de Rham cohomology with compact sipport
%
%
%%%%%%%%%%%%%%%%%%%%%%%%%%%%%%%%%%%%%%%%%%%%%%%%%%%%%%%%%%%%%%%%%%%%%%%%%%%%%%%%%%%%%%%%%%%%%%%%%%%%%%%%%%%%%%%%%%%%%%%%%%%%%%%%%%%%%%%%%%%%%%%%%%%%%%%%%
%
%
\section{de Rham cohomology with compact support}\label{section:caomactly support}
 In this section we define the de Rham cohomology of diffeological spaces with compact support.
 We shall show that there exists a Mayer-Vietoris exact sequence for de Rham cohomology with compact support.
% First,
% we define a compact support.
  \begin{dfn}\label{def:definition of compact support}
   Let $X$ be a diffeological space. 
   Let $\alpha$ be a differential $p$-form on $X$.
   An element $x$ in $X$ is a support element of $\alpha$ if and only if for any plot $P \colon U \to X$ of $X$ and any $r \in U$ such that $P(r)=x$,
   $\alpha(P)(r)$ is nonzero.
% We call the closure of the subset in $X$ of support elements of $\alpha$ the support of $\alpha$,
% it will be denoted by supp$\alpha$,
% that is,
   We call the closure of the set of support elements the support of $\alpha$, 
   and it will be denoted by supp$\alpha$:
    \[
     {\rm supp}\alpha = \overline{ \{ x\in X | \forall P \colon U \to X,\ \forall r \in U \ {\rm s.t} \ P(r)=x,\ \alpha(P)(r) \not = 0 \}}.
    \]
   We say that $\alpha$ is a compactly supported $p$-form on $X$ if the support of $\alpha$ is $D$-compact in $X$.
   The set of compactly supported $p$-forms of $X$ is denoted by $\Omega^{p}_{c}(X)$.
  \end{dfn}
 It is clear that $\{ \Omega^{p}_{c}(X),\,d\}$ is the subcomplex of the de Rham complex $\{\Omega^{p}(X),\, d\}$.
 We define the $p$-th de Rham cohomology group of $X$ with compact support to be the quotient space
  \[
   H^{p}_{c}(X)=Z^{p}_{c}(X)/B^{p}_{c}(X),
%\frac{{\rm Ker}[d \colon \Omega^{p}_{c}(X) \to \Omega^{p+1}_{c}(X)]}{{\rm Im}[d \colon \Omega^{p-1}_{c}(X) \to \Omega^{p}_{c}(X)]}.
  \]  
  where,
   \begin{eqnarray}
    Z^{p}_{c}(X)&=&{\rm Ker}[d \colon \Omega^{p}_{c}(X) \to \Omega^{p+1}_{c}(X)] \ \mbox{and} \nonumber \\
    B^{p}_{c}(X)&=&{\rm Im}[d \colon \Omega^{p-1}_{c}(X) \to \Omega^{p}_{c}(X)]. \nonumber
   \end{eqnarray}
%  $Z^{p}_{c}(X)={\rm Ker}[d \colon \Omega^{p}_{c}(X) \to \Omega^{p+1}_{c}(X)]$ and $B^{p}_{c}(X)={\rm Im}[d \colon \Omega^{p-1}_{c}(X) \to \Omega^{p}_{c}(X)]$.
  We say that a smooth map $f \colon X \to Y$ is a proper map if for any $D$-compact set $C$ in $Y$,
  the inverse image $f^{-1}(C)$ is $D$-compact in $X$.
  Now for any element $\alpha \in \Omega^{p}_{c}(Y)$,
  $f^{\ast}(\alpha)$ is an element of $\Omega^{p}_{c}(X)$ since $f^{-1}({\rm supp} \alpha)$ is $D$-compact in $X$.
  Therefore $f$ induces a homomorphism
   $
    f^{\ast} \colon H^{p}_{c}(Y) \to H^{p}_{c}(X).
   $
  If $X$ is $D$-compact,
  then we have $H^{p}_{c}(X)=H^{p}_{\rm dR}(X)$ since $\Omega^{p}_{c}(X)=\Omega^{p}(X)$ holds.
  
 Let $X$ be a $D$-Hausdorff space.
 Let $A$ be a $D$-open set of $X$.
 If a subset $C$ of $A$ is $D$-compact in $A$,
 then it is $D$-compact in $X$ by Proposition \ref{prp:D-compact equivalent}.
% Thus for any plot $P \colon U \to X$ of $X$,
% $P^{-1}(C)$ is closed in $U$ since $C$ is $D$-closed in $X$.
 %
 Hence the inclusion $i \colon A \to X$ induces a map
  \[
   i_{\ast} \colon \Omega^{p}_{c}(A) \to \Omega^{p}_{c}(X)
  \]
 defined by for any $\alpha \in \Omega^{p}_{c}(A)$,
  \begin{eqnarray}
   i_{\ast}(\alpha)=
   \left\{
    \begin{array}{ll}
     \alpha & on \ A \\
     0 & on \ X \setminus {\rm supp}\alpha, \nonumber 
    \end{array}
   \right.
  \end{eqnarray}
 that is,
 for any plot $P \colon U \to X$ of $X$,
 $i_{\ast}(\alpha)(P)=\alpha(P|_{P^{-1}(A)})$ on $P^{-1}(A)$ and $i_{\ast}(\alpha)(P)=0$ on $P^{-1}(X \setminus {\rm supp}\alpha)$. %$C$ is $D$-closed in $X$. 
%
% Since a compact set $C$ of $A$ is compact in $X$ by Proposition \ref{prp:D-compact equivalent},
% it is $D$-closed in $X$.
% Thus for any plot $P \colon U \to X$,
% the intersection of $P^{-1}(C)$ and the boundary of $P^{-1}(A)$ is empty.
% $P^{-1}(C)$ is a proper subset of $A$ since $P^{-1}(C)$ is closed in $U$.
% $P^{-1}(C)$ and $P^{-1}(A)$ are closed and open in $U$,
% respectively, since $C$ is $D$-closed in $X$.
% Hence the inclusion map $i\colon A \to X$ induces a chain map
%  \[
%   i_{\ast} \colon \Omega^{p}_{c}(A) \to \Omega^{p}_{c}(X),
%  \]
% defined as follows: for any $\alpha \in \Omega^{p}_{c}(A)$ and
% for any plot $P \colon U \to X$ of $X$,
%  \[
%   i_{\ast}(\alpha)(P)=\left\{ 
%    \begin{array}{ll}
%     \alpha(P|U_{A}) & {\rm on} \ U_{A} \\ \\
%     0 & {\rm on } \ U \setminus P^{-1}({\rm supp}\alpha),
%    \end{array}
%   \right.
%  \]
% where $U_A =P^{-1}(A)$.
 Then it is clear that $i_{\ast}$ is injective and $i^{\ast} \circ i_{\ast}(\alpha)=\alpha$ holds.
 Moreover we get a homomomorphism $i_{\ast} \colon H^{p}_{c}(A) \to H^{p}_{c}(X)$.
 %%%%%%%%%%%%%%%%%%%%%%%%%%%%%%%%%%%%%%%%%%%%%%%%%%%%%%%%%%%%%%%%%%%%%%%%%%%%%%%%%%%%%%%%%%%%%%%%%%%%%%%%%%%%%%%%%%%%%%
%
%
% coprod 2
%
%
%%%%%%%%%%%%%%%%%%%%%%%%%%%%%%%%%%%%%%%%%%%%%%%%%%%%%%%%%%%%%%%%%%%%%%%%%%%%%%%%%%%%%%%%%%%%%%%%%%%%%%%%%%%%%%%%%%%%%%%
 \begin{prp}
  Let $X$ be a $D$-Hausdorff space.
  Let $\{X_{\lambda}\}$ be a collection of subspaces of $X$ such that $X$ can be written as a coproduct $X= \coprod_{\lambda \in \Lambda}X_{\lambda}$.
  Then $\oplus_{\lambda \in \Lambda}H^{p}_{c}(X_{\lambda})$ and $H^{p}_{c}(X)$ are isomorphic for each $p$.
 \end{prp}
 \begin{proof}
  Let $\sum i_{\lambda \ast} \colon \oplus_{\lambda \in \Lambda} \Omega^{p}_{c}(X_{\lambda}) \to \Omega^{p}_{c}(X)$
  be the map defined by
   \[
    \left( \sum i_{\lambda \ast}\right)((\omega_{\lambda})_{\lambda \in \Lambda})= \sum_{\lambda \in \Lambda} i_{\lambda \ast}( \omega_{\lambda}),
   \]
   where $i_{\lambda \ast} \colon \Omega^{p}_{c}(X_{\lambda}) \to \Omega^{p}_{c}(X)$ is the chain map induced by the inclusion $i_{\lambda}$.
   Let $(\omega_{\lambda})_{\lambda \in \Lambda}$ be an element of $\mbox{Ker}( \sum i_{\lambda \ast})$.
   For each $\lambda \in \Lambda$ and
   any plot $P \colon U \to X_{\lambda}$ of $X_{\lambda}$,
   it is also a plot of $X$.
   Then we have
    \[
     \left(\sum i_{\lambda \ast}\right)((\omega_{\lambda})_{\lambda \in \Lambda})(P)=
     \sum_{\lambda \in \Lambda} i_{\lambda \ast}(\omega_{\lambda})(P)=\omega_{\lambda}(P)=0.
    \]
   Thus $\sum i_{\lambda \ast}$ is injective since $\omega_{\lambda}=0$ for each $\lambda \in \Lambda$.
%   Let $P \colon U \to X_{\lambda}$ be a plot of $X_{\lambda}$ for any $\lambda \in \Lambda$.
%   Since $i_{\lambda^{\prime} {\ast}}(\omega_{\lambda^{\prime}})(P)=0$ holds for each $\lambda^{\prime} \not = \lambda$,
%    \[
%     \omega_{\lambda}(P)=\sum_{\lambda \in \Lambda} i_{\lambda \ast}(\omega_{\lambda})(P)=(\sum i_{\lambda \ast})((\omega_{\lambda})_{\lambda \in \Lambda})(P)=0.
%    \]
%    Thus $(\omega_{\lambda})_{\lambda \in \Lambda}=0$ holds.
    Next we shall show that $\sum i_{\lambda \ast}$ is surjective.
    Let $\tau$ be an element of $\Omega^{p}_{c}(X)$.
    For each $\lambda \in \Lambda$,
    $X_{\lambda}$ is $D$-open in $X$ by the properties of coproduct diffeology.
    Since supp$\tau$ is $D$-compact in $X$ and $\{X_{\lambda} \}_{\lambda \in \Lambda}$ is a cover of supp$\tau$,
    there exists a finite cover $\{X_{\lambda_{i}}\}_{1 \leq i \leq m}$ of supp$\tau$.
    We define $(\tau_{j})_{j \in \Lambda} \in \oplus_{\lambda \in \Lambda} \Omega^{p}_{c}(X_{j})$ by 
     \[
      (\tau_{j})_{j \in \Lambda}=
       \left\{
        \begin{array}{ll}
         \tau|_{X_{\lambda_{i}}} & j=\lambda_{i} \\
         0 & j \not =\lambda_{i}.
        \end{array}
       \right.
     \]
    Then $(\tau_{j})_{j \in \Lambda}$ is well-defined since $X_{\lambda} \cap X_{\lambda^{\prime}}$ is empty for each $\lambda$ and $\lambda^{\prime}$ in $\Lambda$ such that
    $\lambda^{\prime} \not = \lambda$.
    Clearly,
    we have $(\sum i_{\lambda \ast})((\tau_{j})_{j \in \Lambda})=\tau$.
%   Since $X_{\lambda} \cap X_{\lambda^{\prime}}= \emptyset$ holds for each $\lambda \not = \lambda^{\prime}$,
%   we have $(\sum i_{\lambda \ast})((\tau_{j})_{j \in \lambda})= \tau$.
 \end{proof}
%%%%%%%%%%%%%%%%%%%%%%%%%%%%%%%%%%%%%%%%%%%%%%%%%%%%%%%%%%%%%%%%%%%%%%%%%%%%%%%%%%%%%%%%%%%%%%%%%%%%%%%%%%%%%%%%%%%%%%%%%%%%%%%%%%%%%%%%%%%%%%%%%%%%%%%%%%%%%
%%
% Mayer-Vietoris 2
%%
%%%%%%%%%%%%%%%%%%%%%%%%%%%%%%%%%%%%%%%%%%%%%%%%%%%%%%%%%%%%%%%%%%%%%%%%%%%%%%%%%%%%%%%%%%%%%%%%%%%%%%%%%%%%%%%%%%%%%%%%%%%%%%%%%%%%%%%%%%%%%%%%%%%%%%%%%%%%%%
 \begin{thm}[Mayer-Vietoris exact sequence]\label{thm:Mayer-Vietoris 2}
  Let $X$ be a $D$-Hausdorff space.
  Let $\{A,\,B\}$ be a $D$-open cover of $X$.
  If there exists a partition of unity $\varphi_{i} \colon X \to \mathbf{R} \ (i=A, \,B)$ subordinate to $\{A,\,B\}$,
%partition of unity subordinate $\{\varphi_{i} \colon X \to \mathbf{R}|i=A,B\}$ to $\{A,B\}$,
  then we have a long exact sequence:
   \[
    \to H^{p}_{c}(A \cap B) \xrightarrow{ i_{1 \ast} \oplus i_{2 \ast}}  H^{p}_{c}(A) \oplus H^{p}_{c}(B) \xrightarrow{ j_{1 \ast}-j_{2 \ast}} H^{p}_{c}(X)
    \xrightarrow{\delta}  H^{p+1}_{c}(A \cap B) \to \cdots.
   \] 
%   \[
%   \begin{CD}
%    \dots@> \partial >>
%    H^{p}_{c}(A \cap B)
%    @> i_{1 \ast} \oplus i_{2 \ast}>>
%    H^{p}_{c}(A) \oplus H^{p}_{c}(B)
%    @> j_{1 \ast}-j_{2 \ast}>>
%    H^{p}_{c}(X) \\
%    @.
%    @> \partial >>
%    H^{p+1}_{c}(A \cap B)
%    @>i_{1 \ast} \oplus i_{2 \ast}>>
%    \dots \ \ .
%   \end{CD}
%  \]
 \end{thm}
 \begin{proof}
  To see existence of the Mayer-Vietoris exact sequence,
  it suffices to show that the sequence
   \[
    0 \to \Omega^{p}_{c}(A \cap B) 
    \xrightarrow{ i_{1 \ast} \oplus i_{2 \ast}} 
    \Omega^{p}_{c}(A) \oplus \Omega^{p}_{c}(B)
    \xrightarrow{j_{1 \ast} - j_{2 \ast}}
    \Omega^{p}_{c}(X)
    \to
    0
   \]
%    \[
%   \begin{CD}
%    0@>>>
%    \Omega^{p}_{c}(A \cap B)
%    @> i_{1 \ast} \oplus i_{2 \ast} >>
%    \Omega^{p}_{c}(A) \oplus \Omega^{p}_{c}(B)
%    @>j_{1 \ast} - j_{2 \ast} >>
%    \Omega^{p}_{c}(X)
%    @>>>
%    0
%   \end{CD}
%   \]
  is exact for each $p$.
  It is not difficult to prove that $i_{1 \ast}\oplus i_{2 \ast}$ is injective and that Im$(i_{1 \ast} \oplus i_{2 \ast})$ is a linear
  subspace of Ker$(j_{1 \ast} - j_{2 \ast})$.
  Let $(\alpha, \beta)$ be an element of Ker$(j_{1 \ast}-j_{2 \ast})$.
  Since $j_{1 \ast}(\alpha)=j_{2\ast}(\beta)$ holds,
  $\alpha|_{A \cap B}=\beta|_{A \cap B}$
  and $\alpha=\beta=0$ on $X \setminus ({\rm supp}\alpha \cap {\rm supp} \beta)$.
%  Since supp$(\varphi_{A} \times \omega) \subset A$ and supp$(\varphi_{B} \times \omega) \subset B$, 
  Therefore we have $(i_{1 \ast} \oplus i_{2 \ast})(\alpha|_{A \cap B})=(\alpha,\beta)$.
%  Then for any plot $P \colon U \to X$ of $X$,
%  we have $j_{1 \ast}(\alpha)(P)=j_{2 \ast}(\beta)(P)$.
%  Let $U_{A}=P^{-1}(A),\ U_{B}=P^{-1}(B)$ and $U_{A \cap B}=P^{-1}(A \cap B)=U_{A} \cap U_{B}$.
%  Then we have
%   \[
%    j_{1 \ast}(\alpha)(P)= \left\{
%     \begin{array}{ll}
%      \alpha(P|U_{A \cap B})= \beta(P|U_{A \cap B}) & {\rm on} \ U_{A \cap B} \\ \\
%      0 & {\rm on } \ U \setminus P^{-1}({\rm supp} \alpha).
%     \end{array}
%     \right.
%   \]
%  Let $i^{\ast}_{1} \colon \Omega^{p}_{c}(A) \to \Omega^{p}_{c}(A \cap B)$ be the induced chain map.
 % Then we have $i_{1 \ast} i^{\ast}_{1}( \alpha)=\alpha$ and $i_{2 \ast} i^{\ast}_{1}( \alpha)= \beta$.
%  Therefore $(i_{1 \ast} \oplus i_{2 \ast})(i^{\ast}_{1}(\alpha))=(\alpha, \beta)$.
  We shall show that $(j_{1 \ast} -j_{2 \ast})$ is surjective.
  Let $\omega$ be an element of $\Omega^{p}_{c}(X)$.
  Then $j^{\ast}_{1}(\varphi_{A} \times \omega) \in \Omega^{p}_{c}(A)$ and $j_{2}^{\ast}(-\varphi_{B} \times \omega) \in \Omega^{p}_{c}(B)$.
  Since supp$(\varphi_{A} \times \omega) \subset A$ and supp$(\varphi_{B} \times \omega) \subset B$,  we have
   \begin{eqnarray}
    (j_{1 \ast}-j_{2 \ast})(j_{1}^{\ast}(\varphi_{A} \times \omega),j_{2}^{\ast}(-\varphi_{B} \times \omega ))
    &=&
    j_{1 \ast}j_{1}^{\ast}(\varphi_{A} \times \omega)-j_{2 \ast}j^{\ast}_{2}(-\varphi_{B} \times \omega)) \nonumber \\ \label{dia:shiki3}
    &=&
    \varphi_{A} \times \omega + \varphi_{B} \times \omega \nonumber \\
    &=&
    (\varphi_{A}+\varphi_{B}) \times \omega \nonumber \\
    &=&
    \omega. \nonumber
   \end{eqnarray}
%  where (\ref{dia:shiki3}) is satisfied by the properties of two functions $\varphi_{A}$ and $\varphi_{B}$.
%  Let $\eta_{A}=\phi_{A} \times \omega$ and $\eta=-\phi_{B} \times \omega$.
%  Then $(j_{1}^{\ast}(\eta_{A}),j_{2}^{\ast}(\eta_{B}))$ is an element of $\Omega^{p}_{c}(A) \oplus \Omega^{p}_{c}(B)$.
%  Since $j_{1 \ast} \circ j_{1}^{\ast}(\eta_{A})=\eta_{A}$ and $j_{2 \ast} \circ j_{2}^{\ast}(\eta_{B})=\eta_{B}$,
%  we have
%   \[
%    (j_{1 \ast}-j_{2 \ast})(j_{1}^{\ast}(\eta_{A}),j_{2}^{\ast}(\eta_{B}))=\eta_{A}+\eta_{B}=\phi_{A} \times \omega+\phi_{B} \times \omega=\omega.
%   \]
 \end{proof}
% \begin{rem}
%  Let $X$ be a diffeological subcartesian space with $D$-paracompact.
%  For any $D$-open cover $\sf U$ of $X$,
%  there exists a partition of unity subordinate to $\sf U$ by Theorem \ref{thm:partition of unity}.
%  Therefore we have the Mayer-Vietoris exact sequences of Theorem \ref{thm:Mayer-Vietoris 1} and Theorem \ref{thm:Mayer-Vietoris 2}.
% \end{rem}
 %%%%%%%%%%%%%%%%%%%%%%%%%%%%%%%%%%%%%%%%%%%%%%%%%%%%%%%%%%%%%%%%%%%%%%%%%%%%%%%%%%%%%%%%%%%%%%%%%%%%%%%%%%%%%%%%%%%%%%%%%%%%%%%%%%%%%%%%%%%%%%%%%%%%%%%%%%%%
 %%
 % Long exact sequence
 %%
 %%%%%%%%%%%%%%%%%%%%%%%%%%%%%%%%%%%%%%%%%%%%%%%%%%%%%%%%%%%%%%%%%%%%%%%%%%%%%%%%%%%%%%%%%%%%%%%%%%%%%%%%%%%%%%%%%%%%%%%%%%%%%%%%%%%%%%%%%%%%%%%%%%%%%%%%%%%%
\section{Long exact sequence for pair of diffeological spaces}
 In this section we shall prove Theorem \ref{thm:exact sequence}.
% We introduce the basic properties of a homotopy and a retract.
 Let $f_{0}, \ f_{1} \colon X \to Y$ be two smooth maps between diffeological spaces.
 We say that $f_{0}$ and $f_{1}$ are homotopic if there exists a homotopy $F \colon X \times \mathbf{R} \to Y$ satisfying
 $F(x,0)=f_{0}(x)$ and $F(x,1)=f_{1}(x)$.
 If $f_{0}$ and $f_{1}$ are homotopic then $f^{\ast}_{0}=f^{\ast}_{1} \colon H^{p}_{\rm dR}(Y) \to H^{p}_{\rm dR}(X)$ by {\cite[6.88]{Zem}}.
% A smooth map $F \colon X \times \mathbf{R} \to Y$ satisfying $F(x,0)=f_{0}(x),\ F(x,1)=f_{1}(x)$ is called a homotopy between $f_{0}$ and $f_{1}$.
% By {\cite[6.88]{Zem}}, if there is a homotopy $F$ between $f_{0}$ and $f_{1}$,
% we have 
%  \[
%   f^{\ast}_{0}=f^{\ast}_{1} \colon H^{p}_{\rm dR}(Y) \to H^{p}_{\rm dR}(X).
%  \]

 Let $A$ be a subspace of a diffeological space $X$.
 If there exists a retraction $\gamma \colon X \to A$ such that $\gamma|A=1_{A}$,
 then $A$ is called a retract of $X$.
 Moreover we say that $A$ is a deformation retract of $X$ if $\gamma$ and the identity map $1_{X} \colon X \to X$ are homotopic.
% Then it is clear that $\gamma^{\ast} \colon H^{p}_{\rm dR}(A) \to H^{p}_{\rm dR}(X)$ is an isomorphism.
% Moreover if there exists a homotopy between a retraction $\gamma$ and the identity $1_{X}$,
% $A$ is called a deformation retract of $X$.
  \begin{prp}\label{prp:retract compact}
   Let $A$ and $C$ be subsets of a diffeological space $X$ such that $C \subset A$.
   If there exists a $D$-open set $V$ of $X$ such that $A$ is a retract of $V$,
   then $C$ is $D$-compact in $A$ if and only if $C$ is $D$-compact in $X$.
  \end{prp}
  \begin{proof}
   If $C$ is $D$-compact in $A$,
   then it is $D$-compact in $X$ since all $D$-open sets of $X$ are $D$-open in $A$.
   Conversely,
   let $C$ be a $D$-compact set of $X$.
   Let $\{U_{\lambda}|\lambda \in \Lambda \}$ be a cover of $C$,
   where $U_{\lambda}$ is $D$-open in $A$ for each $\lambda \in \Lambda$.
   Then we have
    \[
     C \subset \cup_{\lambda \in \Lambda} U_{\lambda}= \cup_{\lambda \in \Lambda}(\gamma^{-1}(U_{\lambda}) \cap A) \subset \cup_{\lambda \in \Lambda} \gamma^{-1}(U_{\lambda}),
    \]
   where $\gamma \colon V \to A$ is a retraction.
   Since $\gamma^{-1}(U_{\lambda})$ is $D$-open in $V$,
   it is $D$-open in $X$ by Lemma \ref{lmm:D-open subspace}.
%   Then we have
%    \[
%     A \subset \cup_{1 \leq i \leq m} U_{\lambda_{i}}
%    \]
%   since $C$ is $D$-compact in $X$.
   Since $C$ is $D$-compact in $X$,
   $C \subset \cup_{1 \leq i \leq m} \gamma^{-1}(U_{\lambda_{i}})$ holds.
   Then we have
    \[
     C \subset \cup_{1 \leq i \leq m} ( \gamma^{-1}(U_{\lambda_{i}}) \cap A ) = \cup_{1 \leq i \leq m} U_{\lambda_{i}}.
    \]
   Therefore $C$ is $D$-compact in $A$.
  \end{proof}
%%%%%%%%%%%%%%%%%%%%%%%%%%%%%%%%%%%%%%%%%%%%%%%%%%%%%%%%%%%%%%%%%%%%%%%%%%%%%%%%%%%%%%%%%%%%%%%%%%%%%%%%%%%%%%%%%%%%%%%%%%%%%%%%%%%%%%%%%%%%%%%%%%%%%%%%%%%%%%%%%%%
%
%
%%%%%%% % Long exact Theorrem
%
%
%%%%%%%%%%%%%%%%%%%%%%%%%%%%%%%%%%%%%%%%%%%%%%%%%%%%%%%%%%%%%%%%%%%%%%%%%%%%%%%%%%%%%%%%%%%%%%%%%%%%%%%%%%%%%%%%%%%%%%%%%%%%%%%%%%%%%%%%%%%%%%%%%%%%%%%%%%%%%%%%%%%
 We will prove the following theorem.
% referring to {\cite[p.135]{MJ}}.
 \begin{thm}\label{thm:exact sequence}
  Let $A$ be a $D$-compact set in a diffeological subcartesian space $X$.
  If there exists a $D$-open set $M$ of $X$ such that $A$ is a deformation retract of $M$,
  then we have a long exact sequence:
   \[
   \to H^{p}_{c}(X \setminus A) \xrightarrow{i_{\ast}} H^{p}_{c}(X) \xrightarrow{j^{\ast}}H^{p}_{c}(A) \xrightarrow{\delta} H^{p+1}_{c}(X \setminus A) \xrightarrow{i_{\ast}} \cdots,
   \]
  where $i \colon X \setminus A \to X$ and $j \colon A \to X$ are inclusions.
 \end{thm}
% We shall show the above theorem in the following procedures. 
 We prove this by using the argument similar to that of {\cite[Proposition 13.11]{MJ}}.
 The map $j$ coincides with the composite of the inclusions:
  \[
   A \xrightarrow{k_{1}} M \xrightarrow{k_{2}} X.
  \]
 Since $A$ is $D$-compact by Proposition \ref{prp:retract compact} and it is a deformation retract of $M$,
 a map $\gamma^{\ast} \colon H^{p}_{c}(A)=H^{p}_{\rm dR}(A) \to H^{p}_{\rm dR}(M)$ is an isomorphism,
 where $\gamma \colon M \to A$ is a retraction.
% Since $X$ is $D$-normal by Theorem \ref{thm:normal},
 Moreover,
 there exists a $D$-open set $U_{A}$ of $X$ such that 
  \[
   A \subset U_{A} \subset \overline{U}_{A} \subset M
  \]
 since $X$ is $D$-normal.
% by Theorem \ref{thm:normal}.
% By Corollary \ref{crl:partition},
 Therefore there exists a function $\varphi \colon X \to \mathbf{R}$ such that supp$\varphi \subset M$ and $\varphi \equiv 1$ on $\overline{U}_{A}$ by Corollary \ref{crl:partition}.
 Then we have the following lemma.
 %%%%%%%%%%%%%%%%%%%%%%%%%%%%%%%%%%%%%%%%%%%%%%%%%%%%%%%%%%%%%%%%%%%%%%%%%%%%%%%%%%%%%%%%%%%%%%%%%%%%%%%%%%%%%%%%%%%%%%%%%%%%%%%%%%%%%%%%%%%%%%%%%%%%%%%%%%
 %%%%%%%%%%
%Lemma Three condtion
%%%%%%%%%%%%%
%%%%%%%%%%%%%%%%%%%%%%%%%%%%%%%%%%%%%%%%%%%%%%%%%%%%%%%%%%%%%%%%%%%%%%%%%%%%%%%%%%%%%%%%%%%%%%%%%%%%%%%%%%%%%%%%%%%%%%%%%%%%%%%%%%%%%%%%%%%%%%%%%%%%%%%%%%
 \begin{lmm}\label{lmm:three lemma}
  \begin{enumerate}
   \item\label{item:condition1}
    $j^{\ast} \colon \Omega^{p}_{c}(X) \to \Omega^{p}_{c}(A)$ is surjective.
   \item\label{item:condition2}
    For any $\omega$ in $Z^{p}_{c}(A)$.%{\rm Ker}$[d \colon \Omega^{p}_{c}(A) \to \Omega^{p+1}_{c}(A)]$,
    there exists $\tau$ in $\Omega^{p}_{c}(X)$ such that $j^{\ast}(\tau)=\omega$ and $d\tau |_{U_{A}}$ is zero.
   \item\label{item:condition3}
    For any $\omega$ in $\Omega^{p}_{c}(X)$ such that {\rm supp}$(d\tau) \cap A$ is empty and $j^{\ast}(\omega)$ is zero,
    there exists $\sigma$ in $\Omega^{p-1}_{c}(X)$ such that $(\omega - d\sigma)|_{U_{A}}$ is zero.
  \end{enumerate}
 \end{lmm}
 \begin{proof}
  We shall show the condition {(\ref{item:condition2})}.
  Let $\omega$ be an element of $Z^{p}_{c}(A)$.
%Ker$[d \colon \Omega^{p}_{c}(A) \to \Omega^{p+1}_{c}(A)]$.
  Since $\varphi \times \gamma^{\ast}(\omega) \in \Omega^{p}_{c}(M)$,
  $\tau=k_{2 \ast}(\varphi \times \gamma^{\ast}(\omega))$ is an element of $\Omega^{p}_{c}(X)$.
  Then for any plot $P \colon U \to A$ of $A$,
  we have
   \begin{eqnarray}
    j^{\ast}(\tau)(P)
    &=&
    j^{\ast} k_{2 \ast}( \varphi \times \gamma^{\ast}(\omega))(P) \nonumber \\
    &=&
    (\varphi \times \gamma^{\ast}(\omega))(P) \nonumber \\
    &=&
    \varphi(P) \times \omega(\gamma \circ P). \nonumber 
%   &=&
%    \omega(P), \label{dia:shiki4}
   \end{eqnarray}
  But $\varphi  \equiv 1$ and $\gamma \circ P=P$ on $A$,
  we have $j^{\ast}(\tau)(P)=\omega(P)$.
%  where (\ref{dia:shiki4}) is satisfied since $\varphi \equiv 1 $ on $A$.
  Moreover for any plot $Q \colon V \to X$ of $X$,
  we have
   \begin{eqnarray}
    d \tau |_{U_{A}}(Q) 
    &=&
    d \tau (Q|_{Q^{-1}(U_{A})}) \nonumber \\
    &=&
    d \omega(\gamma \circ Q|_{Q^{-1}(U_{A})}) \nonumber \\
    &=&
    0. \nonumber
   \end{eqnarray}
  Similarly,
  we can prove the condition {(\ref{item:condition1})} in the same argument.
  We shall show the condition {(\ref{item:condition3})}.
  Let $\omega$ be an element of $\Omega^{p}_{c}(X)$ such that {\rm supp}$(d \tau) \cap A= \emptyset$ and $j^{\ast}(\omega)=0$.
  Then $k_{2}^{\ast}(\omega) \in \Omega^{p}(M)$.
  We have 
   \[
    k_{1}^{\ast}[k_{2}^{\ast}(\omega)]=[k_{1}^{\ast}k_{2}^{\ast}(\omega)]=[(k_{2} \circ k_{1})^{\ast}(\omega)]=[j^{\ast}(\omega)]=0.
   \]
  Since $k_{1}^{\ast} \colon H^{p}_{\rm dR}(M) \to H^{p}_{\rm dR}(A)=H^{p}_{c}(A)$ is an isomorphism,
  $[k_{2}^{\ast}(\omega)]=0$ holds.
  Thus there exists $\sigma_{0}$ in $\Omega^{p-1}(M)$ such that $d \sigma_{0}=k^{\ast}_{2}(\omega)=\omega|_{M}$.
  Then $\varphi \times \sigma_{0}$ is in $\Omega^{p-1}_{c}(M)$ and $d(\varphi \times \sigma_{0}) |_{U_{A}}= \omega |_{U_{A}}$
  since {\rm supp}$\varphi \subset M$ and $\varphi \equiv 1$ on $U_{A}$.
  Let $\sigma=k_{2 \ast}(\varphi \times \sigma_{0}) \in \Omega^{p}_{c}(X)$.
  Clearly,
  $(d\sigma - \omega)|_{U_{A}}=0$ holds.
 \end{proof}
Let $\Omega^{p}_{c}(X,A)$ be the kernel of the chain map $j^{\ast} \colon \Omega^{p}_{c}(X) \to \Omega^{p}_{c}(A)$.
Let us denote the cohomology of the subcomplex $\{\Omega^{p}_{c}(X,A),d\}$ by $H^{p}_{c}(X,A)$.
By the property {(\ref{item:condition1})} of Lemma \ref{lmm:three lemma},
we have a short exact sequence:
 \[
  0 \to \Omega^{p}_{c}(X,A) \xrightarrow{l_{\ast}} \Omega^{p}_{c}(X) \xrightarrow{j^{\ast}} \Omega^{p}_{c}(A) \to 0,
 \]
where $l_{\ast}$ is the inclusion map.
Thus we have the following long exact sequence:
 \begin{eqnarray}
  \to H^{p}_{c}(X,A) \xrightarrow{l_{\ast}} H^{p}_{c}(X) \xrightarrow{j^{\ast}} H^{p}_{c}(A) \xrightarrow{\delta} H^{p+1}_{c}(X,A) \to \cdots. \label{dia:exact sequence}
 \end{eqnarray}
 
 Now, we get a cochain map $i_{\ast} \colon \Omega^{p}_{c}(X \setminus A) \to \Omega^{p}_{c}(X,A)$
 since for any $\tau$ in $\Omega^{p}_{c}(X \setminus A)$,
 $j^{\ast} i_{\ast}( \tau)=0$ holds.
 Then we have the following.
  \begin{prp}\label{prp:cohomology equivalence}
   $i_{\ast} \colon H^{p}_{c}(X \setminus A) \to H^{p}_{c}(X,A)$ is an isomorphism.
  \end{prp}
   \begin{proof}
    We shall show that $i_{\ast}$ is injective.
    Let $[\omega]$ be an element of Ker $i_{\ast}$.
    Since $i_{\ast}[\omega]=0$ holds,
    there exists $\tau$ in $\Omega^{p-1}_{c}(X,A)$ such that $d \tau=i_{\ast}(\omega)$.
    Then for any plot $P$ of $A$,
    we get $d \tau(P)=i_{\ast}(\omega)(P)=0$.
    Thus supp$(d \tau) \cap A$ is empty and $j^{\ast}(\tau)$ is zero.
    Then there exists $\sigma$ in $\Omega^{p-2}(X)$ such that 
     \begin{eqnarray}
      (\tau-d \sigma)|_{U_{A}}=0 \nonumber \label{dia:shiki5}
     \end{eqnarray}
    by the property {(\ref{item:condition3})} of Lemma \ref{lmm:three lemma}.
    Hence $\tilde{\tau}=(\tau -d \sigma)|_{X \setminus A}$ is an element of $\Omega^{p-1}_{c}(X \setminus A)$. % by the condition (\ref{dia:shiki5}).
    Hence we have
     \begin{eqnarray}
      d \tilde{\tau} &=& (d \tau - dd\sigma)|_{X \setminus A} \nonumber \\
      &=&
      d \tau |_{X \setminus A} \nonumber \\
      &=&
      i_{\ast}(\omega)|_{X\setminus A} \nonumber \\
      &=&
      \omega. \nonumber  %\label{dia:shiki6} 
     \end{eqnarray}
%     where (\ref{dia:shiki6}) is satisfied since 
%    because we have 
%      \[
%       d \tau |_{X \setminus A}=i_{\ast}(\omega)|_{X\setminus A}=\omega .
%      \]
      Therefore $i_{\ast}$ is injective since $[ \omega]=0$ holds.
      Next,
      we shall show that $i_{\ast}$ is surjective.
      Let $[\omega]$ be an element of $H^{p}_{c}(X,A)$.
      Since $\omega$ in $Z^{p}_{c}(X,A)$, %Ker$[d \colon \Omega^{p}_{c}(X,A) \to \Omega^{p+1}_{c}(X,A)]$,
      there exists $\sigma$ in $\Omega^{p-1}_{c}(X)$ such that $(\omega-d \sigma)|_{U_{A}}=0$ by the property {(\ref{item:condition3})} of Lemma \ref{lmm:three lemma}.
      Since we have
       \[
        d(j^{\ast}(\sigma))=j^{\ast}(d \sigma)=j^{\ast}(\omega)=0,
       \]
      there exists $\tau$ in $\Omega^{p-1}_{c}(X)$ such that $j^{\ast}(\tau)=j^{\ast}(\sigma)$ and $d \tau|_{U_{A}}=0$ by the property {(\ref{item:condition2})} of Lemma \ref{lmm:three lemma}.
      Then $\sigma - \tau$ is an element of $\Omega^{p-1}_{c}(X,A)$ since $j^{\ast}(\sigma - \tau)=0$ holds.
      Let $\tilde{\omega}=(\omega-d(\sigma-\tau))|_{X \setminus A}=(\omega - d \sigma)|_{X\setminus A}+d \tau|_{X\setminus A}$.
      Then $\tilde{\omega}$ in $\Omega^{p-1}_{c}(X \setminus A)$ and we get
       \[
        i_{\ast}[\tilde{\omega}]=i_{\ast}[(\omega - d \sigma)|_{X \setminus A}]=[i_{\ast} i^{\ast}(\omega-d \sigma)]=[\omega-d \sigma]=[\omega].
       \]
       Therefore $i_{\ast}$ is surjective.
   \end{proof}
   Therefore we have Theorem \ref{thm:exact sequence} by the exact sequence (\ref{dia:exact sequence}) and Proposition \ref{prp:cohomology equivalence}.
  \begin{crl}
   Let $X$ be a $D$-compact diffeological subcartesian space. % with $D$-compact.
   Let $A$ be a $D$-closed subset of $X$.
   If there exists a $D$-open subset $M$ of $X$ such that $A$ is a deformation retract of $M$,
   then we have a long exact sequence:
    \[
     \to H^{p}_{c}(X \setminus A) \xrightarrow{i_{\ast}} H^{p}_{\rm dR}(X) \xrightarrow{j^{\ast}} H^{p}_{\rm dR}(A) \xrightarrow{\delta} H^{p+1}_{c}(X\setminus A) \to \cdots.
    \]
  \end{crl}
  \begin{proof}
   Clearly,
   $H^{p}_{c}(X)=H^{p}_{\rm dR}(X)$ and $H^{p}_{c}(A)=H^{p}_{\rm dR}(A)$ since $A$ is $D$-compact by Proposition \ref{prp: subcompact} and Proopsition \ref{prp:retract compact}.
   Therefore the exactness of the sequence above follows from Theorem \ref{thm:exact sequence}. %it is clear that we have the above long exact sequence.
  \end{proof}

\end{document}